\documentclass[11pt]{amsart}
\usepackage{amsmath,amsfonts,amsthm,amssymb,amscd,color,xcolor,mathrsfs}
\usepackage{amsfonts}
\usepackage{amsmath,amsthm}
\usepackage{hyperref}
\usepackage{latexsym}
\usepackage{array}
\usepackage{amssymb}
\usepackage{enumerate}

\usepackage[francais,english]{babel}

\usepackage{color}
\usepackage[latin1]{inputenc}

\usepackage[margin=1.2in]{geometry}

\numberwithin{equation}{section}

\usepackage{amsfonts}
\usepackage{latexsym}
\usepackage{amsmath,amsthm,amssymb, color, amscd,marginnote}
\usepackage[all]{xy}
 \usepackage{mathrsfs}
\usepackage{MnSymbol}
\usepackage{marginnote}
\usepackage{hyperref}
\hypersetup{colorlinks, linkcolor=blue,citecolor=blue}

\usepackage{graphicx,psfrag,epsfig}

\theoremstyle{plain}
\newtheorem{theorem}{Theorem}

\newtheorem{lemma}[theorem]{Lemma}
\newtheorem{corollary}[theorem]{Corollary}

\newtheorem{proposition}[theorem]{Proposition}
\theoremstyle{definition}
\newtheorem{definition}[theorem]{Definition}
\newtheorem{remark}[theorem]{Remark}

\def\eps{\varepsilon}

\def\lan{\langle}
\def\ran{\rangle}

\def\bdef{\begin{definition}}
\def\endef{\end{definition}}
\def\bthm{\begin{theorem}}
\def\ethm{\end{theorem}}
\def\blm{\begin{lemma}}
\def\elm{\end{lemma}}
\def\brm{\begin{remark}}
\def\erm{\end{remark}}
\def\bprop{\begin{proposition}}
\def\eprop{\end{proposition}}
\def\bcor{\begin{corollary}}
\def\ecor{\end{corollary}}
\def\be{\begin{eqnarray}}
\def\ee{\end{eqnarray}}
\def\beal{\begin{aligned}}
\def\enal{\end{aligned}}

\def\om{\omega}

\def\eps{\varepsilon}

\def\phi{\varphi}

\def\R{\mathbb R}

\def\N{\mathbb N}
\def\T{\mathbb T}

\def\Z{\mathbb Z}

\def\cP{\mathcal P}

\def\~{\tilde}

\def\cD{\mathcal D}

\def\EE{\mathbf{E}}

\def\p{\partial}
 \newcommand{\strela}{\rightharpoonup}

\def\llan{\langle\!\langle}
\def\rran{\rangle\!\rangle}

\def\be{\begin{equation}}
\def\ee{\end{equation}}
\def\bdef{\begin{definition}}
\def\endef{\end{definition}}
\def\blm{\begin{lemma}}
\def\elm{\end{lemma}}
\def\beal{\begin{aligned}}
\def\enal{\end{aligned}}

\newtheorem*{Pf}{Proof}
\renewenvironment{proof}{\begin{Pf} \begin{upshape}} {\end{upshape} \qed\end{Pf}}

\numberwithin{equation}{section}
\numberwithin{theorem}{section}

\title[Kolmogorov's 4/5-law for Burgers equation]
{Weak and strong versions of the Kolmogorov 4/5-law for stochastic Burgers equation}

\author{Peng Gao}
\address{School of Mathematics and Statistics, and Center for Mathematics and Interdisciplinary Sciences, Northeast Normal University,
Changchun 130024,  P. R. China,\& Peoples' Friendship University of Russia (RUDN University), Moscow, Russia}
\email{gaop428@nenu.edu.cn}	

\author{Sergei Kuksin}
\address{Universit\'e Paris Cit\'e and Sorbonne Universit\'e, CNRS, IMJ-PRG, F-75013 Paris, France 
   \&  Peoples' Friendship University of Russia (RUDN University),
6 Miklukho-Maklaya St, Moscow, 117198, Russian Federation}
\email{ sergei.kuksin@imj-prg.fr}

\begin{document}

\maketitle

\begin{abstract}
For solutions of the space-periodic stochastic 1d Burgers equation we establish two versions of the Kolmogorov 4/5-law which provides an asymptotic
expansion for the third moment of increments of turbulent velocity fields. We also prove for this equation an analogy of the Landau objection 
to possible universality of Kolmogorov's  theory of turbulence, and show that the third moment is the only one which admits a universal
asymptotic expansion.
 \end{abstract}

\date{}

	\maketitle
	%\centerline { \today}
%	
%	\tableofcontents

\section{The 4/5 law} The Kolmogorov theory of turbulence, known as the K41 theory (see in \cite{LL, Fr}), examines homogeneous turbulence,
corresponding to velocity fields $u(t,x)$ which are random fields, stationary in time $t$, homogeneous and isotropic in the space-variable $x$.
Using dimension analysis and arguing on physical level of rigour, Kolmogorov made a number of remarkable predictions concerning
small-scale properties of  velocity fields  with large Reynolds numbers $R$, corresponding to increments
$
u(t,x+r)-u(t,x).
$
Here the vectors $r$ are such that their lengths $|r|$ belong to the {\it inertial range},  which is an interval in $\R_+$, formed by  real numbers which are ``small but not too small" in term of $R$ and the rate of dissipation of energy
 $\eps =\nu \EE | \nabla u(t,x)|^2$,
where $\nu>0$ is the kinematic viscosity of the fluid.  One of these predictions is the 4/5-law, stating
that for large $R$ and for $r$ from the inertial range
\be\label{K45}
\EE \Big[ \big(u(t,x+r)-u(t,x)\big)\cdot \tfrac{r}{|r|}\Big]^3 = - \tfrac45\, \eps |r|.
\ee
%Here
%$\eps = \EE | \nabla u(t,x)|^2$
%ùalmost independent from large $R$.
 Later the law was intensively discussed by physicists and was re-proved, using  physical arguments, always  related to the
 original Kolmogorov's arguing (see \cite[Sec.~6.2]{Fr} and \cite[Sec.~2.2.2]{Fal}). Rigorous verification of the 4/5-law and other
 laws of the K41 theory remains an outstanding open mathematical problem. Recently a progress
  was achieved in \cite{Bedr}. There -- as it is often the case since 1960's -- 3d
 turbulent flows are modelled by solutions of the stochastic 3d NSE with small viscosity $\nu>0$
  on the torus $\T^3$. It is known that the latter equation has stationary solutions. Taking such a solutions
  $u^\nu(t,x)$ and  assuming  that it meets the assumption
  $
  \nu \EE\| u(t)\|_{L_2}^2 =o(1) \,
  $
  as $\nu\to0$ they prove that \eqref{K45} holds after averaging in $r$ over a sphere of  radius $|r|$.
  The relation  is established    for all $|r|$ from an interval in $\R_+$
  whose left end goes to zero with $\nu$, but whose relation with the inertial range is   not clear.

 In order to understand better turbulence and the laws  of the K41 theory, starting
1940's physicists use stochastic 1d models, where  the most popular one is given by the stochastic space-periodic Burgers equation, see
\cite{FF, BF, BKh}. The equation describes fictitious 1d ``burgers fluid" and turbulence in it, called by U.~Frisch {\it burgulence}.
Our goal in this work is to rigorously derive for this equation two relations, which may be regarded as weak and strong forms of the 4/5-law
\eqref{K45}.

In  the next Section \ref{s_2}  we discuss the stochastic  space-periodic  Burgers equation, following the book \cite{BK}.
There we develop the notation and state properties
of the equation's solutions, used in the rest of the paper.
Then in Section~\ref{s_3} we prove a weak form of the 4/5-law for the
 Burgers equation. There we show that for any solution $u(t,x)$ of the equation, its cubed increments
$
(u(t,x+l) -u(t,x))^3
$
with $l$ from the inertial range for burgulence,
averaged in $x$, in ensemble and locally averaged in time, behave as $-$Const$\,l$, uniformly in small non-negative viscosities. See relations
\eqref{B451}, \eqref{B452}.

In Section \ref{s_4} we prove a strong form of the 4/5-law for the Burgers equation. Namely, there in Theorem~\ref{t_final} we show that if $u^{\nu \,st}(t,x)$ is a
stationary in time solution of the equation with a positive
viscosity $\nu$\ \footnote{All such solutions have the same distribution.}
and $\eps^B =\nu \EE \int | u_x(t,x)|^2dx$
 is its  rate of dissipation of energy,   then, for any $t$,
\be\label{final}
 \EE \int\big( (u(t,x+l) -u(t,x)\big)^3 dx =-12 \eps^B l+ o(l) \quad \text{as} \;\; \nu\to0.
\ee 
The relation is proved to hold for  $l$ from a  {\it strongly inertial range}, which is ``just a bit smaller" than the inertial range for burgulence. The latter is the segment $ [c_1\nu, c]$, where the constants $c_1, c_2$ depend on the force, applied to the
``burgers fluid" (see a discussion after Theorem~\ref{t_4}), and the strongly inertial range  is defined in  \eqref{strong}. 
 In Corollary~\ref{c_4.2} we show that
\eqref{final} as well
 holds for any solution of Burgers equation, asymptotically as $t\to\infty$.  We also prove in Theorem~\ref{t_final} that
a stationary solution of the inviscid Burgers equation satisfies \eqref{final} with removed limit ``$\nu\to0$" and $o(l)$
replaced by $O(l^3)$.
Relation \eqref{final} is the form in which the ``4/5-law for burgulence" appears in works of physicists, justified by  heuristic arguments. E.g.
see \cite{Pol, DV}.

In  Section \ref{s_5} we discuss the criticism of possible universality of the K41 theory, made by  Landau.
 It  implies (on the physical level of rigour)  that in the frame of  K41
 the only universal relation for moments of increments $u(t,x+r) - u(t,x)$ is the 4/5-law \eqref{K45}
for cubic moments. We state a reformulation  for burgulence of the Landau claim and rigorously prove it.

%Concerning some other laws of K41 and their 1d analogies for Burgers equation see \cite[Chapter 6]{BK} and \cite{Kuk}.
\medskip

\noindent {\it Notation.} For a metric space $M$ we denote by $\cP(M)$ the set of probability Borel measures on $M$, for a random variable $\xi$,
valued in $M$, denote by $\cD(\xi)\in \cP(M)$ its distribution, and for a function $f$ and a measure $\mu$ on $M$ denote
$
\lan f,\mu\ran = \int_M f\,d\mu.
$
The arrow $\strela$  stands for   weak convergence of measures.

\section{Stochastic Burgers equation}\label{s_2}
\subsection{The setting and well posedness of the equation}
The stochastic  Burgers  equation under periodic boundary conditions has the following form:
\begin{eqnarray}\label{1.1}
u_{t}(t,x)+u(t,x)u_{x}(t,x)-\nu u_{xx}(t,x)=\partial_{t}\xi(t,x),\qquad t\geq0, \; x\in S^{1}:=\mathbb{R}/\mathbb{Z},
\end{eqnarray}
where the viscosity coefficient
 satisfies $\nu\in(0,1]$, and
 \be\label{1.2}
 \xi(t,x):=\sum\limits_{s\in \mathbb{Z}^{*}}b_{s}\beta_{s}(t)e_{s}(x).
 \ee
 Here $\{b_{s}\}$ are real numbers,  $\{\beta_{s}\}$ are standard independent Wiener processes defined on a probability space
  $(\Omega,\mathcal{F},\mathbb{P})$, and $\{ e_s(x), s \in \Z^* =\Z\setminus\{0\}\}$ is an
   usual trigonometric basic in the space of
  1-periodic functions with zero mean-value. Namely,  for a $k\ge1$,
  $
  e_k(x) = \sqrt2\, \cos(2\pi kx)
  $
  and
  $
  e_{-k}(x) = \sqrt2\, \sin(2\pi kx).
  $
We suppose that the process $\xi$ is non-zero and is sufficiently smooth in $x$:
\be\label{B5}
B_0>0, \;\; B_ {M}<\infty \;\;\text{ for some } \; M\ge5, \;\; \text{where} \;\; B_ {m}:= \sum |2\pi s|^{2m} b_s^2\le\infty.
\ee
Since always $\nu\le1$, then below $\nu>0$ stands for $0<\nu\le1$.

 As the space-meanvalue of $\xi(t,x)$ is zero, then $\int_{S^1} u(t,x)\,dx$ is an integral of motion for the equation, and we always assume that it vanishes:
  $$
  \int u(t,x)\,dx\equiv 0.
  $$

For short we will write the solutions $u^\nu(t,x)$ of \eqref{1.1} as  $u^\nu(t)$ or just $u(t)$. If $u_0$ is a random field,  independent from $\xi$
 (e.g. $u_0$ is non-random), then by  $u^\nu(t,x;u_0)= u^\nu(t;u_0)$ we will denote a solution of \eqref{1.1}, equal $u_0$ at $t=0$. By $H$ we
 denote the space of $L_2$-functions on $S^1$ with zero mean, given the $L_2$-norm $\|\cdot\|$ (so $\{ e_s\}$ is its Hilbert basis). For
 $m\in \N \cup \{0\}$ we denote by $H^m$ the Sobolev space
 $$
 H^m =\{ v\in H: \p^m v\in H\},
 $$
 given the homogeneous norm $\| u\|_m = \| \p^m u\|$ (so $\| u\|_0 = \|  u\|$). By $| \cdot |_p$, $1\le p\le\infty$, we denote the $L_p$-norm on
 $S^1$, and abbreviate $L_p(S^1)$ to $L_p$.

 It is well known that under assumption \eqref{B5} eq.~\eqref{1.1} is well posed in  spaces $H^m$, $1\le m\le M$, and defines in spaces  $H^m$ with $1\le m\le M-1$ Markov processes with continuous trajectories.
 E.g. see \cite{DZ, BK}
 and references in \cite{BK}. Moreover, if $u_0\in H^1$ is non-random, then
 \be\label{es1}
 \EE \exp (\sigma \| u^\nu(t;u_0)\|^2) \le C
 \ee
 for some $\sigma, C>0$ (depending on $u_0, \nu$ and the random force). If $u_0 \in H^m$, $1 \le m\le M$, then %a.s. $u^\nu(t;u_0) \in C([ 0,\infty), H^{\min(m, M-1)}),$ and
 \be\label{es2}
 \EE \| u^\nu(t; u_0)\|_m^2
  \le C_m  \quad \forall\, t\ge0,
 \ee
 where again $C_m$ depends on $u_0, \nu$ and $\xi$.  A
  solution $u^\nu= u^\nu(t;u_0) $ may be constructed by  Galerkin's method. Then $u^\nu$ is obtained as a limit of Galerkin's
 approximations
 $
 u^{(N)}(t) \in $ span$\,\{e_s, |s| \le N\} \subset H,
 $
 which also satisfy estimates \eqref{es1}, \eqref{es2}. If $u_0 \in H^m$ with $1\le m <M$, then a.s.
 \be\label{Gal}
 u^{(N)}(t) \to u^\nu(t) \;\; \text{in} \;\;H^{m-1}\;\; \text{as \ $N\to\infty$, \ uniformly \, in} \; t\in [0,T],
 \ee
 for any finite $T>0$. See in \cite{BK} Theorems 1.4.2 and 1.4.4.

 \subsection{Main estimates for solutions} A remarkable property of solutions $u^\nu$ of eq.~\eqref{1.1}
 is given by
 Oleinik's estimates: If $u_0 \in H^1$ is
 a r.v.,  independent from $\xi$, then for each $0< q <\infty$ and $0<\theta \le1$ there exists $C=C(q, B_4) >0$ such that for every
 $\nu >0$ and $t\ge \theta$  the solution $u^\nu(t)= u^\nu(t;u_0) $ satisfies
 \be\label{O1}
 \EE |u_x^{\nu\, +} (t) |_\infty^q \le C \theta^{-q},
 \ee
  \be\label{O2}
 \EE \big( |u^{\nu} (t) |_\infty^q  +  |u_x^{\nu} (t) |_1^q  \big)
 \le C \theta^{-q}
 \ee
(in \eqref{O1} $u_x^{\nu\, +} (t) =u_x^{\nu\, +} (t,x) $ is a positive part of the function $ u_x^{\nu} (t,x) $).  We stress
 that $C$ does not depend on $\nu$ and $u_0$.
These relations imply crucial lower and upper estimates on Sobolev norms of solutions. To state them we need a definition: for a random process
$f^\om(t)$ and $\sigma\ge0,  T>0$  we denote
\be\label{aver}
 \llan f\rran :=
\llan f\rran_T^{T+\sigma} =  \frac1{\sigma} \int^{T+\sigma}_T \EE f(s)\,ds.
\ee
The following result is proved in \cite[Section 2.3]{BK}:

\begin{theorem}\label{t_1}
For each $\theta>0$, any $\N \ni m\le M$
  and for every random variable $u_0\in H^1$, independent from $\xi$,  solution
$u^\nu(t) =u^\nu(t;u_0)$ satisfies

1) $\EE \| u^\nu(t) \|_m^2 \le C'_m \nu^{-(2m-1)}$ \, for $t\ge \theta$, where $C'_m$ depends on $\theta$ and
$B_{\max(4,m)}$.

2) There exists $T_*(\theta) >0$ such that
\be\label{Hm}
C_m^{-1} \nu^{-(2m-1)} \le \llan \EE \| u^\nu \|_m^2\rran \le C_m \nu^{-(2m-1)}
\ee
for some $C_m\ge1$. Here the averaging
$   \llan \cdot\rran=\llan \cdot \rran_T^{T+\sigma} $  corresponds to arbitrary constants $\sigma\ge\theta$ and  $T\ge T_*$.
The constants  $C_m$ depend on $\theta$, $T_*$ and $\xi$, but not
on  $\sigma\ge\theta, T\ge T_*$ and $\nu\in(0,1]$.
\end{theorem}

We will write $A\sim B$ if the two quantities satisfy
$
C_1 B \le A\le C_2 B
$
for some $C_1, C_2>0$  which do not depend on $\nu$ and $\sigma, T$ as in the theorem. Then \eqref{Hm} may be written as
$$
 \llan \EE \| u^\nu \|_m^2\rran \sim \nu^{-(2m-1)}.
$$
For $m=0$ relations for the norm $\| u^\nu(t;u_0)\|_0$ are different. Namely from \eqref{O2} it follows that
$\EE \| u^\nu(t) \|^2 \le C'_0 $ \, for $t\ge \theta$, and it is shown in \cite{BK} that for the brackets as in the theorem's  assumptions,
$
 \llan \EE \| u^\nu \|^2\rran \sim 1.
$
The quantity $\frac12 \EE \| u^\nu(t) \|^2$ is the (averaged) energy of the ``1d flow" $u^\nu(t,x)$. Formally applying Ito's formula to $\frac12\| u^\nu(t) \|^2$
and taking the expectation we get the {\it balance of energy} relation
\be\label{be}
 \tfrac12 \EE \| u^\nu(t) \|^2 -  \tfrac12 \EE \| u^\nu(0) \|^2 =
 - {\EE}\int_0^t \nu \| u^\nu(s)\|_1^2 ds + \tfrac12 B_0 t.
\ee
See \cite[Chapter 1.4]{BK} for its rigorous derivation if $u_0$ is sufficiently smooth.

 \subsection{The mixing} \label{s_13}
  Equation \eqref{1.1} is mixing.  It means that there exists a unique measure $\mu_\nu \in \cP(H^{M+1})$ such
 that for any r.v. $u_0 \in H^1$, independent from $\xi$,
 \be\label{mix}
 \cD u^\nu(t;u_0) \strela \mu_\nu \;\; \text{in} \;\; \cP(H^{M-1}) \;\; \text{as} \;\; t\to\infty,
\ee
see \cite[Chapter 3.3]{BK}. The rate of convergence in \eqref{mix} may depend on $\nu$ (but  it becomes
$\nu$-independent if we regard $\cD u^\nu(t)$ and $\mu_\nu$ as measures in $L_p$, see \cite[Chapter 4.2]{BK}).
 If $u_0$ is a r.v. such that $\cD u_0 =\mu_\nu$, then
$
\cD u^\nu(t; u_0) \equiv \mu_\nu.
$
Such solutions $u^\nu$ are called {\it stationary}.

 \subsection{Inviscid limit} When $\nu\to0$, a solution $u^\nu(t;u_0)$ converges to a limit, known as an {\it entropy solution} of  eq.~\eqref{1.1}$\mid\!_{\nu=0}$.
 This result may be established in a number of different ways. In the form, given in the theorem below, it is proved in
 \cite[Chapter~8]{BK}, following Kruzkov's approach \cite{Kruz}, based on a version of Oleinik's estimates \eqref{O1}, \eqref{O2}.

\begin{theorem}\label{t_3}
Let $T>0$, $u_0\in H^2$ be a non-random function and $u^\nu= u^\nu(t;u_0)$. Then there exists a random field  $u^0(t,x):=u^{0\,\om}(t,x;u_0)$ such
that almost surely, for any $1\le p<\infty$,

1) $u^0 \in L_\infty( [0,T] \times S^1) \cap  C([0,T];  L_p) $;

2) $u^\nu(t) \to u^0(t)$ in $L_p$, uniformly in $t\in [0,T]$;

3) for any $t\ge\theta>0$ and $0< q<\infty$,\
$
\EE | u^0(t)|_p^q \le C(q, B_4) \theta^{-q};
$

4) $u^0(0,x) = u_0$ and $u^0(t,x)$ satisfies eq. \eqref{1.1} in the sense of generalised functions.
\end{theorem}

Eq. \eqref{1.1}$\mid\!_{\nu=0}$ with a prescribed initial data has many generalised solution, but its entropy solution, defined by the limiting
construction above, exists and is unique. 
Entropy solutions   extend to a mixing process in $L_1$:

\begin{theorem}\label{t_33}
Solutions
 $u^0(t;u_0)$ extend by continuity in $u_0$
  to a Markov process in $L_1$ such that  for every $u_0 \in L_1$, $u^0(t;u_0)$ a.s. is continuous in $t$ and
  meets the estimate in  item 3)
 of Theorem~\ref{t_3}. This process is mixing: there is a measure $\mu_0 \in \cP(L_1)$, satisfying
 $\mu_0( \cap_{q<\infty} L_q)=1$,   such that for every r.v.   $u_0 \in L_1$, independent from $\xi$,
\be\label{0_mix}
\cD u^0(t; u_0) \strela \mu_0 \quad \text{in} \quad \cP(L_p) \quad \text{as}\;\; t\to \infty,
\ee
for any $p<\infty$.
If $\cD u_0 = \mu_0$, then solution $ u^0(t; u_0) $ is stationary:
$ \cD u^0(t; u_0)  \equiv \mu_0$.
\end{theorem}

See \cite[Chapter 8.5]{BK} and see \cite{Sin, BKh} for another approach to stationary solutions of the inviscid
Burgers equation \eqref{1.1}.

\section{Moments of  increments $u^\nu(t, x+r) - u^\nu(t,x)$ and a weak form of the 4/5-law for Burgers equation.}\label{s_3}
Let us adopt some more notation. For a function $v(x)$ on $S^1$ and $|l|<1$ we denote
$
v^l = {v(x+l)}
$
and set
$
\delta^l v(x) = v^l(x) - v(x).
$
The absolute moments of increments in $x$ of
 a solution $u^\nu(t,x)$ for \eqref{1.1} with respect to the brackets  $\llangle\cdot \rrangle$ and  averaging in $x$  are
$$
\llan \int | \delta^l u^\nu(t,x)|^p dx\,\rran  =
\llan\, | \delta^l  u^\nu(t)|_p^p\,\rran =: S_{p,l} =  S_{p,l}(u^\nu)<\infty, \quad p>0.
$$
Obviously if $u^\nu(t)$ is a stationary solution, then
$
S_{p,l}(u^\nu) = \EE  \int | \delta^l u^\nu(t,x)|^p dx\,.
$
The function
$
(p, l) \mapsto  S_{p,l}(u^\nu)
$
is called the {\it structure function} of a solution $u^\nu$.  Since the function
$
v(x) \mapsto | \delta^l  v |_p^p
$
is continuous on $L_{\max(1,p)}$, then in view of item 3) of Theorem~\ref{t_3}
 the structure function $S_{p,l}(u^0)$  for entropy solutions $u^0(t,x)$ also is well defined and finite.

Careful analysis of solutions $u^\nu$ with $\nu\ge0$
 and of estimates \eqref{O1},  \eqref{O2},  \eqref{Hm} implies  that for any random initial data
 $u_0\in H^1$ (independent from $\xi$) the structure function $S_{p,l}(u^\nu)$ obeys the following law:

\begin{theorem}\label{t_4}
There exist constants $c_*, c >0$ and  $c_1\ge1$, and for each $p>0$ there exists $C_p\ge1$, all
depending only on the random force in  eq. \eqref{1.1} and on $\theta, T_*$ as in Theorem~\ref{t_1},
 such that for each
$\nu \in [0, c_*]$ and $p>0$ we have:

%1) if $l\in [c_1\nu, c]$, then
1) if $l\in [c_1\nu, c]$, then
\be\label{s1}
C_{p}^{-1}  l^{\min(1,p)}\le   S_{p,l}(u^\nu) \le C_{p}\, l^{\min(1,p)};
\ee

%2) if $l\in [0, c_1\nu)$, then
2) if $l\in [0, c_1\nu)$, then
$$
C_{p}^{-1}   l^p \nu^{1-\max(1,p)} \le     S_{p,l}(u^\nu) \le C_{p}\,  l^p \nu^{1-\max(1,p)}
$$
(for $\nu=0$ this assertion is empty).
\end{theorem}

We see from this result that statistically the increments $| \delta^lu^\nu(t)|$, $\nu>0$, with
$l\in [0, c_1\nu]$ behave  ``linearly in $l$", while for $l\in [c_1\nu, c]$ their behaviour is non-linear. So
the interval $ [0, c_1\nu]$ is the {\it dissipation range} for burgulence, described by the Burgers equation \eqref{1.1}, while
$ [c_1\nu, c]$ is the {\it inertial range}. The frontier $c_1\nu$  between the two interval is the {\it dissipation} or {\it inner} scale of the flow.
The constants $c_1$ and $c$, depending on the random force, may change from one group of results to another.

On the physical level of rigour the first assertion  of the theorem
was proved in \cite{FFF}. Rigorously for $\nu>0$  it was established in \cite{Bor}, using some ides from \cite{FFF}.
 For a complete proof of assertions  1) and 2) see \cite[Chapter~7.2]{BK}.
 \smallskip

Now we start to discuss moments (not absolute ones) of increments $\delta^l u^\nu$ of solutions  $u^\nu(t;u_0)$
  as above:
\be\label{S3}
S^s_{p,l}(u^\nu) = \llan s_{p,l}(u^\nu(t))\rran, \quad  s_{p,l} (v) =
 \int (\delta^l  v^\nu(x))^p dx, \quad \nu>0, \; \;0 \le l<1,
\ee
where $p\in\N$ (the upper index $s$ stands for ``skew"). If $p$ is an even number then $S^s_{p,l} = S_{p,l}$, but for an odd $p$
the two moments are different. As $S^s_{1,l}=0$, then the first non-trivial  skew moment is the third one.  Let us
examine it.  Since any real number $x$ may be written as $x= 2x^+ - |x|$, then
 \be\label{s3}
 S^s_{3,l} = - S_{3,l} +2 \llan \int \big( \delta^l u^\nu)^+\big)^3 dx \rran.
 \ee
 But for any $x$,
 $
 \big(\delta^l u^\nu(x)\big)^+ \le \int_x^{x+l} (u_x^\nu)^+ (t,y) dy \le l\, | {(u_x^\nu)}^+|_\infty.
 $
 So in view of \eqref{O1} the second term in the r.h.s. of \eqref{s3} is bounded by $Cl^3$, uniformly in $\nu>0$. From here, \eqref{s3}
 and \eqref{s1} with $p=3$ we get that for a suitable $c'>0$ and all $\nu\in(0, c_*]$,
 \be\label{B451}
 -C'_1 l \le  S^s_{3,l}(u^\nu) \le -C'_2 l
  \quad \text{if} \quad l\in [c_1\nu, c'],
 \ee
 for some $C'_1 \ge C'_2>0$, independent from $\nu>0$ and solution $u^\nu$.

 Since in view of \eqref{O2} with $p=4$ the family of functions
 $
 \int (\delta^l  u^\nu(t,x))^3dx$, $\nu>0,
 $
 is uniformly integrable in variables
 $
 (t, \om) \in [\sigma, \sigma +T]\times \Omega,
 $
 then passing to the limit in \eqref{B451}   as $\nu\to0$ using item 2) of
 Theorem~\ref{t_3} with $p=3$ we find that  entropy
 solutions  $u^0= u^0(t;u_0)$, $u_0\in H^2$,  satisfy
 \be\label{B452}
  -C'_1 l \le  S^s_{3,l}(u^0) \le -C'_2 l
  \quad \text{if} \quad l\in [0, c'].
 \ee
 Relation \eqref{B451} + \eqref{B452},
  valid for all small $\nu\ge0$, is a weak form  {of} the 4/5-law \eqref{K45} for Burgers equation. Literally the same argument    shows that the two relations  hold for all moments $ S^s_{p,l}$ with odd $p\ge3$ (the
  constants $C'_1, C'_2$ should be modified).

 In the next section we will see that some ideas, originated in K41, allow to strengthen  the asymptotic
 behaviour \eqref{B451} for $S^s_{3,l}$  (as $\nu\to0$)  with  $l$  in the inertial range to a real asymptotic, if $l$ belongs not to the whole inertial range, but to some large part of the latter.
% but to some of its large part.
 %In a remark, concluding the paper, we explain why it is not likely that the corresponding Kolmogorov's idea
 %may serve to get (real) asymptotic for $S_{p,l}$  or  $S^s_{p,l}$  with $p\ne3$, instead of the asymptotic behaviour like \eqref{s1}.

 \section{Strong form of the 4/5-law for Burgers  {equation}}  \label{s_4}
 Assuming that in \eqref{K45} the velocity field $u$ is homogeneous and isotropic
 in $x$ (not necessarily stationary in $t$), for any $p\in\N$ define the $p$-th moment $ S_p^{||}(t,r)$ of a longitudinal increment
 $
 (u(t,x+r)-u(t,x))\cdot(r/|r|)
 $
 as
 $ \EE \big[ \big(u(t,x+r)-u(t,x)\big)\cdot ({r}/{|r|})\big]^p$
 (so the l.h.s. of \eqref{K45} is $S_3^{||}(t,r)$).  Following Kolmogorov,
  proofs of the 4/5-law in physical works, e.g. in \cite{Fr, Fal}, as well as in the rigorous paper \cite{Bedr},\footnote{There, to adjust the
  formula to periodic boundary conditions, it is integrated in $dr$ with suitable densities. It turns out that thus obtained ``weak KHM formula" is
  sufficient for a ``conditional" derivation of relation \eqref{K45}.}
  crucially use the Karman--Howard--Monin  formula  (which is rather a class of formulas, see for them
   the references above and relation (5.5.5) in   \cite{Bat}). The formula
    relates  time-derivative of the second moment
    $S_2^{||}(t,r)$ with derivatives in $r$  of the third moment  $S_3^{||}(t,r)$. Variations of the  formula (e.g. see in \cite{Fal})   instead  of
      the second moments
  $S_2^{||}$ analyse the   correlations $\EE u(t,x) \cdot u(t, x+r)$, closely related to $S_2^{||}$.
  Thus motivated let us examine time-derivatives
 of   correlations of a solution  $u^ \nu(t,x) = u^ \nu(t,x;u_0)$ of Burgers equation with $\nu>0$ and a non-random initial data
  $u_0\in H^{M-1}$, i.e. of
 $$
 \int u^\nu(t,x) u^{\nu l} (t,x) \,dx =: f^l(u^\nu(t)).
 $$
  Abbreviating $u^\nu(t) $ to $u(t) $, formally  applying  the Ito formula to $f^l(u(t))$ and taking the expectation we get:
 \be\label{I1}
 \begin{split}
 \frac{d}{dt} \EE f^l(u(t)) = \EE\big( -df^l(u)(uu_x) +\nu\, df^l(u)(u_{xx}) &+ \tfrac12 \sum b_s^2 d^2 f^l(u)(e_s,e_s)\big) \\
 & =: \EE( -I_1(t) +I_2(t) +I_3(t)).
 \end{split}
 \ee
 Since
 $
 df^l(u)v =\int (uv^l+u^lv)dx  =\int (uv^l+u v^{-l})dx
 $
 and
 \be\label{as}
 (\p/ \p l) u^l(x) = u_x^l(x),
 \ee
 we get that
 $$
 I_1(t) = \int (u u^l u_x^l + uu^{-l} u^{-l}_x)dx = \tfrac12 \tfrac{\p}{\p l} \int\big (u (u^l)^2 - u(u^{-l})^2\big) dx=
 \tfrac12 \tfrac{\p}{\p l} \int \big(u (u^l)^2 - u^l u^2\big) dx.
 $$
 Recalling that the functional $s_{3,l}$ was defined in \eqref{S3},  we have
 $\
 s_{3,l} (v(x)) = \int \big( \delta^l v(x)\big)^3dx  = 3 \int \big( v^l v^2 - (v^l)^2 v \big)dx.
 $
 So
 $$
 I_1(t) = - \tfrac16 \tfrac{\p}{\p l} s_{3,l} (u(t)) .
 $$

 Similar, using \eqref{as} we get
 \[
 I_2(t) %=\nu \int\!\! \big( uu^l_{xx} + u^l u_{xx}\big)dx
   =\nu \int\!\! \big( uu^l_{xx} + u u^{-l} _{xx}\big)dx
  = \nu \frac{\p^2}{\p l^2} \int\!\! \big( u u^l + u u^{-l}\big) dx = 2\nu  \frac{\p^2}{\p l^2} f^l(u(t)).
 \]

 Since $d^2 f^l (u)(v,v) = 2 f^l(v)$, then relation \eqref{I1} may be re-written as
 \be\label{I2}
  \frac{d}{dt} \EE f^l(u(t))  = \tfrac16 \EE \frac{\p}{\p l} s_{3,l}(u(t)) + 2\nu \EE  \frac{\p^2}{\p l^2} f^l(u(t)) + \tilde B_0(l),
 \ee
where
$\
\tilde B_0(l) := \sum_s b_s^2 f^l(e_s) = \sum_s b_s^2 \cos(2\pi sl).
$

We have obtained \eqref{I1} by a formal application of Ito's formula to the infinite-dimensional stochastic process
 $u^\nu(t;u_0) \in H^{M-1}$ with $u_0 \in  H^{M-1}$ .
  But Galerkin's approximations $u^{(N)}(t)$ to solutions $u^\nu(t) $ satisfy finite-dimensional stochastic systems. Estimates
\eqref{es1}, \eqref{es2} also hold for them and imply the validity of Ito's formula for $u^{(N)}$'s. The latter
 has the form \eqref{I1} with Ito's term $\EE I_3(t)$  modified to $\EE  \tfrac12 \sum_{|s|\le N} b_s^2 d^2 f^l(u)(e_s,e_s)$.
 Then passing to a limit as $N\to\infty$ using \eqref{Gal} and  the uniform in $N$ estimates we justify the validity of \eqref{I1}
for $u(t) =u^\nu(t;u_0)$. (Doing that we write the Ito
 equation in the integrated in time form.)  Cf. \cite{BK}, where the energy balance \eqref{be}
is established in a similar way. Since solutions $u^\nu(t) \in H^{M-1}$ meet estimates \eqref{es1}, \eqref{es2} with $m=M-1\ge4$, then the given above
formal transformation from \eqref{I1} to \eqref{I2} also is rigorous.

 Relation  \eqref{I2}  is a version of the Karman--Howarth--Monin formula for
the stochastic Burgers equation.

Now let $\mu_\nu \in \cP(H^{M+1})$ be the stationary measure for eq. \eqref{1.1} (see Section~\ref{s_13}), and let $u^{\nu \,st}(t)$ be a corresponding
stationary solution. Then
%$\cD u^{\nu \,st}(t)  \equiv \mu_\nu $
 $u^{\nu \,st}(t)= u^\nu(t;u_0)$, where  $\cD u_0 = \mu_\nu$.
%is a stationary  {solution},
Using estimate \eqref{O2}, where $u^\nu=u^{\nu \,st}$,
 we see that all terms in the integrated in time
 relation \eqref{I2} with  $u(t) = u^\nu(t; u_0)$, are integrable in $\mu_\nu(du_0)$. Performing this integration we get that  equality
 \eqref{I2} stays true for  $u=u^{\nu \,st}(t) $. Then the l.h.s. of \eqref{I2} vanishes, so the relation takes form
 $$
 (\p /\p l) \EE \big(s_{3,l}( u^{\nu \,st}(t) )\big) = -12 \nu  (\p^2 /\p l^2) \EE \big(f^l (u^{\nu \,st}(t) )\big) -6 \tilde B_0(l).
 $$
 Since $s_{3,0}( u(x))\equiv0$ and since by \eqref{as}
 $
 (\p/\p l) f^l(u)\!\mid\!_{l=0} = \int u(x) u_x(x) dx =0,
 $
 then integrating this equality in $dl$ we find that
 \be\label{I3}
  \EE \big(s_{3,l}( u^{\nu \,st}(t) )\big) = -12 \nu  (\p /\p l) \EE \big(f^l (u^{\nu \,st}(t) )\big) -6 \int_0^l \tilde B_0( {r}) d {r}.
 \ee
 Next, convergence \eqref{mix}, estimate  \eqref{O2}, item 1) of Theorem~\ref{t_1}  and Fatou's lemma imply that
  for all $\nu>0$,
 \be\label{I4}
 \EE \| u^{\nu  st} (t)\|_1^2 \le C\nu^{-1},
 %\lan \| u\|^2 ,\mu_\nu\ran \le C, \quad \lan \| u\|^2_1 ,\mu_\nu\ran \le \nu^{-1}C
 \ee
 for all $\nu>0$ and a suitable $C$.
 Consider the first term in the r.h.s of \eqref{I3}. Dropping   the factor  $-12\nu$  and using \eqref{as} we write its modulus as
 \[
 \begin{split}
\big| \EE \int u u_x^l dx\big|  &=  \big| \EE \int (u(t,x) - u(t,x+l)) u_x(x+l) dx\big|  \\
  &\le
 \Big[ \EE \int\big( u(t,x) -u(t,x+l)\big)^2dx\Big]^{1/2}  \Big[ \EE \int  u_x(t,x)^2dx\Big]^{1/2}
 \end{split}
 \]
 (we used that $\int u(x+l) u_x(x+l)dx=0$).
 Since $u$ is a stationary solution, then the first factor in the r.h.s. equals $S_{2, l}^{1/2}(u^{\nu \,st})$. So in view of
 Theorem~\ref{t_4} and estimate \eqref{I4} the first term in the r.h.s. of \eqref{I3} is
 $
 O\big(\sqrt{l} \sqrt\nu\big).
 $

 In view of \eqref{B5}, $\tilde B_0$ is an even $C^2$-function. Since $\tilde B_0(0) = B_0$, then
 $
 \int_0^l \tilde B_0( {r}) d {r} = B_0 l+ O(l^3).
 $
 Using in \eqref{I3} the estimates for the terms in its r.h.s. which we have just obtained we find that
 \be\label{I5}
  \EE \big(s_{3,l}( u^{\nu \,st}(t) )\big) = -6 B_0 l+ O(l^3)  +  O\big(\sqrt{l} \sqrt\nu\big).
 \ee

 By relation (8.5.4) in \cite{BK},  $\mu_\nu \strela \mu_0$ in $\cP(L_3)$ as $\nu\to0$.  Next,  by convergence \eqref{mix} and estimate \eqref{O2}
 (with $q=4$),
 $
 \lan |u|_3^4, \mu_\nu \ran \le C
 $
 uniformly in $\nu>0$ (e.g. see \cite[Corollary~11.1.7]{BK}).
 Since functional $s_{3,l}$ is continuous on $L_3$ and $s_{3,l} (u) \le C |u|_3^3$, then we derive from here that
 $$
 \lim_{\nu\to0} \lan s_{3,l}, \mu_\nu\ran = \lan s_{3,l}, \mu_0\ran .
 $$
 Let $u^{0\, st} (t) $ be a stationary entropy solutions of eq. \eqref{1.1}$\mid\!_{\nu=0}$,  $\cD( u^{0\, st} (t)) \equiv \mu_0$ (see Theorem~\ref{t_33}).
 Then
 $
 \lan s_{3,l}, \mu_0\ran =
 \EE s_{3,l}( u^{0\, st} (t) ).
 $
 So passing in \eqref{I5} to the limit as $\nu\to0$, we get that
 \be\label{I6}
  \EE \big(s_{3,l}( u^{0\, st}(t) )\big) = -6 B_0 {l} + O(l^3).
   \ee

   If in \eqref{I5} $l$ belongs to the inertial range $[c_1\nu, c]$, then the norm of third term in the r.h.s. of \eqref{I5} is
   bounded by $C c_1^{-1/2} l$. Assuming that $c_1$ is sufficiently big, we obtain from  \eqref{I5} another proof of
    the weak law \eqref{B451} for stationary
   solutions $u^{\nu \,st}(t)$ (since
   $
   S^s_{3,l}(u^{\nu \,st}) = \EE  s_{3,l}(u^{\nu \,st}(t)) ).
   $

   Now let $l$ belongs to a ``strongly inertial range", i.e. 
   \be\label{strong}
   l \in [L(\nu), c],% \quad \text{where} \quad L(\nu) \to 0 \;\;\text{and}\;\;
   % L(\nu) /\nu \to \infty\;\; \text{as}\; \nu\to0,
   \ee
   for any fixed function $L(\nu)$ such that
   $$
    L(\nu) \to 0 \;\;\text{and}\;\;
    L(\nu) /\nu \to \infty\;\; \text{as}\; \nu\to0.
   $$ 
 Then $\sqrt{l} \sqrt\nu =o(l)$ as $\nu\to0$ and  we arrive at the main  result of this work:

\begin{theorem}\label{t_final}
Let $ u^{\nu \,st}(t) $, $\nu>0$, be a stationary solution of eq. \eqref{1.1} and $l$ satisfies \eqref{strong}. Then
 \be\label{I55}
 S^s_{3,l}( u^{\nu \,st}(t) )=
  \EE \big( s_{3,l}( u^{\nu \,st}(t)) \big) = -6 B_0 l+ o(l) \;\; \text{as} \; \nu\to0,
 \ee
 where  $o(l)$ depends only on the function $L(\nu)$ 	and the random force $\xi$. While
 the stationary entropy solution $ u^{0\, st}(t) $ satisfies \eqref{I6}.\footnote{Equivalently,
$
\lan   s_{3l}, \mu_0\ran  = -6 B_0  l+ O(l^3)
$ as $l\to0$. }
\end{theorem}

Due to the balance relation \eqref{be}, for  stationary solution $ u^{\nu \,st}$ with $\nu>0$ the rate of dissi\-pation of energy is given by
\be\label{epsB}
\eps^B = \tfrac12 B_0.
\ee
% $B_0$ is the double rate of the solution's dissipation of energy $\eps^B$.
  So \eqref{I55} may be written as
\be\label{better}
  S^s_{3,l} ( u^{\nu \,st}(t) ) = - 12 \eps^B l + o(l)  \quad \text{as} \quad \nu \to 0.
\ee
%This is the form in which the ``4/5 law for Burgers equation" appears in works of physicists, e.g. see \cite{DV}.

Combining the theorem's result with \eqref{mix} and \eqref{0_mix} we get

\begin{corollary}\label{c_4.2} 1) Let $\nu>0$, and $u_0\in H^1$ be a r.v., independent from $\xi$. Then for any $l$
as in  \eqref{strong} we have
$$
\lim_{t\to\infty} \EE \big(s_{3,l}( u^{\nu }(t;u_0) )\big) = -6 B_0 l+ o(l) \;\; \text{as} \; \nu\to0.
$$

2) If $\nu=0$ and   $u_0\in L_1$ is a r.v., independent from $\xi$, then
$$
 \lim_{t\to\infty} \EE \big(s_{3,l}( u^{0}(t;u_0) )\big) = -6 B_0 {l} + O(l^3) \;\; \text{as} \; l \to0.
$$
\end{corollary}

The proof of Theorem  \ref{t_final} crucially uses that the moment  which we analyse is cubic. Indeed, for the proof the Ito term $I_3$ should be a constant, for
that the functional $f^l$ should be quadratic, and then the  term $I_1$ in \eqref{better} is of the third order in $u$.
But this does not imply that moments $S^s_{p,l}$
with integers $p\ne 3$ do not admit asymptotic expansions in $l$. In the next section  we show that an asymptotic expansion of $S^s_{p,l}$ with an integer
$p\ge2$, $p\ne3$,
is not possible if  in addition we require that its leading term   ``is universal".

\begin{remark}
 While suitable analogies of Theorems \ref{t_1}, \ref{t_3}, \ref{t_4} and of the weak 4/5-law \eqref{B452} hold for
solutions of the free Burgers equation \eqref{1.1}${}_{\xi=0}$ with a non-zero smooth initial data
 (see \cite[Section~10.11]{BK}), we see no way to establish for the latter
equation a reasonable analogy of Theorem~\ref{t_final}.
\end{remark}

%\section{The $J_0$ term: case $d=4$}
\section{On the Landau objection to universality in K41 and burgulence}  \label{s_5}
%\cite{Garb}
The celebrated 2/3-law of K41 states that for turbulent velocity fields $u(t,x)$ as those, treated by the theory,
the second moments $S_2^{||}(r)$\,\footnote{The moments do not depend on $t$ by the assumed stationarity
of $u$.}
of longitudinal increments of $u$ behave as $(\eps |r|)^{2/3}$. Originally Kolmogorov insisted on the universality of
the law and claimed that
$
S_2^{||}(r) = C^K (\eps |r|)^{2/3} + o\big((\eps |r|)^{2/3} \big),
$
where $C^K$ is an absolute constant. But this universality  was put in doubt by Landau who suggested a
physical argument,
implying that a relation for a moment of velocity increment may be universal only if the
value of the moment, suggested by the relation,
 is linear in the rate of energy dissipation $\eps$ (like relation \eqref{K45} for the third moment).
See in \cite{LL} a footnote at page 126 and see  \cite[Section~6.4]{Fr}. The goal of this section is to show that  for
burgulence, indeed, the only universal relation for the moments $S^s_{p,l}$  is relation
\eqref{better} for the cubic one (which is linear in $\eps^B$).

 Namely, for a stationary solution $u^{\nu \,st}(t,x)$ of stochastic Burgers equation \eqref{1.1} and an integer $p\ge2$  consider the following relation for the  $p$-th moment $S^s_{p,l}$   of  increments of $u^{\nu \,st}$:
\be\label{L}
 S^s_{p,l}( u^{\nu \,st}(t) )= C_* (\varepsilon^B l)^{q}+ o(\varepsilon^B l)^{q} \;\; \text{as} \; \nu\to0,
 \ee
where $l$ is any number from  the inertial range $[c_1\nu, c]$ and  $q >0$.
 We address the following question: for which
$p$ and $q$ relation \eqref{L} holds with a \textit{universal} constant $C_*$,
independent from  the random force $\xi$?

\begin{theorem}\label{t_L}
%Let $ u^{\nu \,st}(t) $, $\nu>0$, be a stationary solution of eq. \eqref{1.1} and $l$ satisfies \eqref{strong}.
  If relation \eqref{L} holds for any random force $\xi$, satisfying \eqref{B5},  with a $C_*$, independent from $\xi$,
   then
$$p=3,~q=1,~C_*=-12.$$
\end{theorem}

\begin{proof}
Let us abbreviate $u^{\nu \,st}(t)$ to $u(t)$. We  take some real number
 $\mu >1$ and set $\tilde{\xi}(\tau):=\mu ^{-\frac{1}{2}}\xi(\mu \tau)$. This
 also is a standard Wiener process. Denote $w(\tau,x):=\mu u(\mu \tau,x).$ Then $w$ is a stationary
 solution of  equation
\begin{eqnarray}\label{L1}
w_{\tau}(\tau,x)+w(\tau,x)w_{x}(\tau,x)- {\nu^\mu} w_{xx}(\tau,x)=\mu^\frac{3}{2} \partial_{\tau}\tilde{\xi}(\tau,x),
\qquad {\nu^\mu}=\nu \mu.
\end{eqnarray}

Consider the inertial range $J^1 = [c_1\nu, c]$ for eq.  \eqref{1.1} and   inertial
 range  $J^\mu = [c^\mu_1\nu, c^\mu]$  for eq.~\eqref{L1}. For small $\nu$ their intersection
 $ J= J^1\cap J^\mu$ is not empty. For $l\in J$ relation  \eqref{L} holds for $u$ which solves eq.~\eqref{1.1} and for
 $w$, solving  eq.~\eqref{L1}. Since
 $
 S_{p,l}(w) = \mu^p S_{p,l} (u)
 $
 and as  $\eps^B_w = \mu^3 \eps^B_u$ in view of \eqref{epsB},   then from here
 $$
 \mu^p  C_* \big( \eps_u^B l\big)^q + o(\eps_u^B l )^q=   C_* \big( \mu^3 \eps_u^B l\big)^q + o(\eps_u^B l )^q
 $$
 for $l\in J$ and all small $\nu$. As $\mu>1$, then by this equality   $q=p/3$.\,\footnote{
 This is in line with the relation $| u(t,x+r) -u(t,x)| \asymp (\eps |r|)^{1/3}$ which appears in the theory of
 turbulence due to a basic dimension argument, without any relation to the equations, describing  the fluid.
 See \cite[(32,1)]{LL}. }
 On the other hand, it follows from Theorem~\ref{t_4} if $p$ is even and from \eqref{B451} and a discussion after
   \eqref{B452} if $p$ is odd that
$ S_{p,l}( u)\sim -l$ for any integer $p\ge2$. Thus in \eqref{L} $q=1$, and so
 $p=3q=3$. Then by Theorem~\ref{t_final}
 $C_*=-12$ 	and the theorem is proved.
\end{proof}

\begin{remark}
1) The result of Theorem \ref{t_L} remains true with the same proof if relation \eqref{L} is claimed to
hold not for all $l$ from the inertial range, but only  for $l$ from a strongly
inertial range as in \eqref{strong}. In this form asymptotic  \eqref{L}  with $p=3$ and $q=1$ indeed is valid by Theorem~\ref{t_final}.

2) We do not know if for some integer $p\ge2$, different from 3, asymptotical expansion for $S_{p,l}( u^{\nu \,st}(t) )$ of the form
 \eqref{L}, valid for all $l$ from the inertial  range (or from a strongly inertial  range)  may hold  with a constant $C_*$ which  depends on the random force
$\xi$.
\end{remark}

\section*{ Acknowledgement} The authors are thankful to A.~Boritchev for discussion.
PG was supported by Natural Science Foundation of Jilin Province (Grant No. YDZJ202201ZYTS306) and  the Fundamental Research Funds for the Central
Universities. Both authors were supported by the Ministry of Science and Higher Education of the Russian Federation (megagrant No. 075-15-2022-1115).

\section*{ Data Availability Statement}
Data sharing is not applicable to this article as no
datasets were generated or analysed during the current study.

\bibliography{reference}{}
\bibliographystyle{plain}

\end{document}